\documentclass[11pt,a4,leqno]{amsart}

\usepackage{amsmath,amssymb,amsfonts,amsthm,enumitem}
\setlist[enumerate]{label=(\roman*)}
\usepackage{hyperref}
\hypersetup{
	colorlinks = true,
	linkcolor = black,
	citecolor = black
}
\usepackage{tikz}
\usepackage{graphicx}
\usetikzlibrary{arrows.meta}

\newtheorem{thm}{Theorem}[section]
\newtheorem{lemma}[thm]{Lemma}
\newtheorem{prop}[thm]{Proposition}
\newtheorem{cor}[thm]{Corollary}
\newtheorem{question}{Question}
\newtheorem{notation}[thm]{Notation}
\theoremstyle{definition}
\newtheorem{df}[thm]{Definition}
\newtheorem{example}[thm]{Example}
\theoremstyle{remark}
\newtheorem{rem}[thm]{Remark}
\numberwithin{equation}{section}
\theoremstyle{plain}
\newcounter{theoremintro}
\newtheorem{thmx}[theoremintro]{Theorem}
\newtheorem{corx}[theoremintro]{Corollary}

\DeclareMathOperator{\Cu}{Cu}

\newcommand{\Nb}{\mathbb{N}}
\newcommand{\Zb}{\mathbb{Z}}
\newcommand{\Qb}{\mathbb{Q}}
\newcommand{\Cb}{\mathbb{C}}
\newcommand{\Tb}{\mathbb{T}}
\newcommand{\Rb}{\mathbb{R}}

\newcommand{\fF}{{\sf{F}}}
\newcommand{\BD}{{\sf{BD}}}

\renewcommand{\P}{\mathcal{P}}

\begin{document}
      
\begin{abstract}
    We introduce nonnuclear Bunce--Deddens algebras, defined as reduced crossed product $C^*$-algebras associated with free odometer actions of nonamenable countable discrete residually finite groups on a Cantor space. These are nonnuclear counterparts to the generalised Bunce--Deddens algebras introduced by Orfanos. For large classes of groups, we show that they share many regularity properties, including real rank zero, stable rank one, strict comparison of positive elements, and admitting a unique tracial state. In many cases, we deduce that they are even selfless and hence pure. Finally, we compute the ordered $K$-theory and trace pairing of Bunce--Deddens algebras over nonabelian free groups.
\end{abstract}


 
\title{Nonnuclear Bunce--Deddens algebras}
\author{Jamie Bell}
\address{Jamie Bell, Mathematical Institute, University of M\"unster, Einsteinstrasse 62, 48149 M\"unster, Germany.}
\email{jbell@uni-muenster.de}

\thanks{Funded by the Deutsche Forschungsgemeinschaft (DFG, German Research Foundation) under Germany's Excellence Strategy EXC 2044/2 –390685587, Mathematics M\"unster: Dynamics-Geometry-Structure and project-ID 427320536, SFB 1442, of the DFG}

\maketitle

\section{Introduction}

In 1975, John W. Bunce and James A. Deddens constructed the first known examples of simple unital nuclear $C^*$-algebras that are not AF \cite{BunDed75}. These $C^*$-algebras, now known eponymously as Bunce--Deddens algebras, were originally defined as certain simple quotients of algebras of weighted shift operators, but admit a natural dynamical description as the crossed product $C^*$-algebras associated with $\Zb$-odometer actions on a Cantor space. 

Bunce--Deddens algebras are closely related to UHF algebras. Indeed, just as UHF algebras are classified by their supernatural numbers, as proved by Glimm \cite{Gli60}, Bunce--Deddens algebras are classified up to isomorphism by the supernatural numbers associated with the odometers from which they are constructed.\footnote{This is not a coincidence; as later proved by Putnam \cite{Put89}, every Bunce--Deddens algebra contains an orbit-breaking UHF subalgebra of the same supernatural number, and this inclusion induces an order isomorphism of $K_0$-groups.} They emerged just as the rich infusion of $K$-theoretic methods into operator algebras was unfolding \cite{Ell76,PimVoi80,Rie81}. From a modern viewpoint, Bunce--Deddens algebras provide tractable examples of simple separable nuclear $C^*$-algebras that are classified by their Elliott invariant, belonging to the class of A$\Tb$ algebras with real rank zero later classified by Elliott and Elliott--Gong \cite{Ell93,EllGon96}. We refer to \cite{Win19,Whi22} for surveys on the current status and history of the Elliott classification programme.

Let us briefly recall the classical construction of $\Zb$-odometers. Take a strictly decreasing sequence $n_1\Zb\ge n_2\Zb \ge \cdots$ of finite-index subgroups of $\Zb$ -- equivalently, a strictly increasing sequence of natural numbers $n_i \in \Nb$ such that $n_i \mid n_{i+1}$ for all $i \in \Nb$. The associated odometer is the action $\Zb \curvearrowright \varprojlim \Zb/n_i\Zb$ given by coordinatewise left-translation on the profinite completion of $\Zb$ determined by the chosen subgroups. Writing $v_p(n_i)$ for the $p$-adic valuation of $n_i$ at a prime $p$, the crossed product associated with this action is (isomorphic to) the Bunce--Deddens algebra whose supernatural number is the formal product 
\[
    \mathfrak{n} = \prod_{p} p^{m_p}, \quad  m_p=\sup_i v_p(n_i)\in \Zb_{\ge 0}\cup\{\infty\}. 
\]
Orfanos \cite{Orf10} introduced the nuclear generalised Bunce--Deddens algebras. Mutatis mutandis, these are obtained from the odometer construction above by replacing the chains of finite-index  subgroups of $\Zb$ with chains of finite-index normal subgroups of amenable residually finite groups with trivial intersection -- such chains we call normal and separating -- and taking the crossed products of the resulting odometer actions (Definition~\ref{D:BD}). Building on the work of Phillips \cite{Phi05}, Orfanos proved that these $C^*$-algebras retain many of the regularity properties of the classical Bunce--Deddens algebras; they are simple, nuclear and quasidiagonal, and have real rank zero, stable rank one, comparison of projections, and a unique tracial state.

Bunce--Deddens algebras and their nuclear generalisations have found a number of applications, demonstrating their enduring utility. In addition to providing further examples of Elliott-classifiable $C^*$-algebras \cite{Car11}, they serve, in some cases, as counterexamples to Matui's HK conjecture \cite{Sca20,Dee23}, and satisfy a number of orbit equivalence rigidity results \cite{CorMed16,GioPutSka19}. Odometer-like actions were also used to give the first examples of minimal topologically free actions of amenable groups with zero mean dimension that are not almost finite \cite{Jos24b} (cf.~\cite{GarGefGesKopNar24}), and they appear implicitly in the construction of $C^*$-diagonals in Cuntz algebras \cite{EviSib26}.  

In light of their importance in the study of nuclear $C^*$-algebras, in this article we record what happens when the same construction of generalised Bunce--Deddens algebras is applied to odometer actions of nonamenable residually finite groups, with a particular emphasis on the case of nonabelian free groups. While the methods used by Orfanos -- which rely crucially on the amenability of the acting group in the form of F{\o}lner tilings -- break down in our setting, we show that these algebras nevertheless frequently retain a number of key regularity properties. The clearest picture emerges for Bunce--Deddens algebras over free groups. 

\begin{thmx}\label{T:A}
    Fix $d\ge 2$. Let $\sigma$ be a separating normal chain in $F_d$. Then the Bunce--Deddens algebra $\BD (F_d,\sigma)$ is a simple, separable, unital, nonnuclear, non-$\mathcal{Z}$-stable $C^*$-algebra with real rank zero, stable rank one and a unique tracial state. This algebra is moreover selfless (in the sense of Robert \cite{Rob25}), and hence pure (in the sense of Winter \cite{Win12}). 
\end{thmx}

Variations of Theorem~\ref{T:A} may also be derived for a much larger class of groups from results in Section~\ref{S:regularity}, though we note that not all of the properties hold (or are known to hold) for every such group. The proof of Theorem~\ref{T:A} primarily relies on a well-known decomposition of Bunce--Deddens algebras into inductive limits of matrix amplifications of reduced group $C^*$-algebras, combined with a number of recent results surrounding the study of pureness and selflessness, following breakthroughs in \cite{Rob25,AmrGaoKunPat25,Oza25}. 

While parts of Theorem~\ref{T:A} are likely known to some experts, we hope that our presentation makes these examples also transparent to nonspecialists. One point of emphasis throughout our article is that, while Bunce--Deddens algebras will always be simple, several arguments do not rely on simplicity; accordingly, we formulate a number of our results without imposing freeness or minimality assumptions. This is in stark contrast to the amenable case. 

For Bunce--Deddens algebras associated with free groups, we use the Pimsner--Voiculescu exact sequence to provide a complete computation of the ordered $K$-theory and its trace pairing.

\begin{thmx}\label{T:B}
    Fix $d\ge 2$. Let $\sigma = (G_n)_{n\in\Nb}$ be a separating normal chain in $F_d$. Write $A = \BD(F_d,\sigma)$ and  
    \[
		\mathbb{Q}_\sigma = \bigcup_{n\in\Nb} \frac{1}{[F_d : G_n]} \Zb \subseteq \Qb.    	
    	\]
    	 Then
    \[
        (K_0(A),K_0(A)^+, [1]_0) \cong (\Qb_\sigma,\Qb_\sigma \cap [0,\infty), 1), \quad K_1(A) \cong \Zb^{(\Nb)}.
    \]
    Under the first identification, the trace pairing $\tau_* : K_0(A) \to \Rb$ induced by the unique tracial state $\tau$ on $A$ is the inclusion $\Qb_\sigma \hookrightarrow \Rb$. 
\end{thmx}

\begin{corx}\label{C:C}
    For each $d\ge 2$, there exist continuum many pairwise nonisomorphic Bunce--Deddens algebras over $F_d$. 
\end{corx}

By results of Cortez and Medynets \cite{CorMed16}, the Bunce--Deddens algebras $\BD (G,\sigma)$ and $\BD (H,\rho)$ are isomorphic via an isomorphism preserving their canonical Cantor-spectrum diagonals if and only if the odometer actions $G\curvearrowright \widehat{G}_\sigma$ and $H\curvearrowright \widehat{H}_\rho$ are structurally conjugate (see \cite[Definition 3.1]{CorMed16}); this is also equivalent to the existence of an abstract isomorphism of the topological full groups $\fF (G,\widehat{G}_\sigma)$ and $\fF (H,\widehat{H}_\rho)$. For $\Zb$-odometers, structural conjugacy is equivalent to flip conjugacy. It seems interesting to determine to what extent the isomorphism classes of nonnuclear Bunce--Deddens algebras (forgetting the position of the diagonal) can be understood.  

\medskip
\noindent\textit{Structure}. In Section~\ref{S:prelims}, we cover preliminaries on odometer actions and record the decomposition theorem for their crossed products. In Section~\ref{S:construction}, we define the nonnuclear Bunce--Deddens algebras and record some of their basic properties. Section~\ref{S:regularity} establishes the regularity properties and proves Theorem~\ref{T:A}. Finally, in Section~\ref{S:Kthy} we compute the ordered $K$-theory and trace pairing, and prove Theorem~\ref{T:B} and Corollary~\ref{C:C}. 

\medskip
\noindent\textit{Acknowledgements.} I am particularly grateful to Julian Kranz, discussions with whom inspired the present article. I thank Sven Raum and Jonathan Taylor for their hospitality during a visit to the University of Potsdam, where ideas related to this work were discussed. I am also indebted to Julian Kranz and Spyros Petrakos for their valuable comments on an earlier draft. Finally, I thank the anonymous referee for their careful reading of the manuscript and for providing a number of useful comments. 

\section{Preliminaries}\label{S:prelims}

\subsection{Group actions} An \emph{action} $G \stackrel{\alpha}{\curvearrowright} X$ of a countable discrete group $G$ on a compact Hausdorff space $X$ is a group homomorphism $\alpha : G \to \operatorname{Homeo}(X)$. For $g\in G$ and $x\in X$, we write $\alpha_gx := \alpha(g)(x)$, or simply $gx$ when the action is understood. The action $G\curvearrowright X$ is \emph{minimal} if there exist no nonempty proper closed subsets $Y\subseteq X$ such that $GY := \{gy : g\in G, y\in Y\}\subseteq Y$; equivalently, every orbit $Gx := \{gx : g\in G\}$ is dense in $X$. We write $X/G$ for the \emph{orbit space}, which is the quotient of $X$ by the relation $x\sim y$ if and only if $Gx = Gy$, $x,y\in X$. The action is \emph{(topologically) transitive} if for all nonempty open sets $U,V\subseteq X$ there exists $g\in G$ such that $gU\cap V\ne\emptyset$. If $X$ is finite, $G\curvearrowright X$ is transitive if and only if for every $x,y\in X$ there exists $g\in G$ such that $gx = y$. For $x\in X$, we write $G_x:= \{g\in G: gx = x\}$ for the stabiliser subgroup of $G$. The action is \emph{free} if $G_x = \{e\}$ for all $x\in X$. Two actions $G \curvearrowright X$ and $G \curvearrowright Y$ are \emph{conjugate} if there exists a homeomorphism $\varphi : X\to Y$ such that $\varphi(gx) = g\varphi(x)$ for all $g\in G$ and $x\in X$. We write $M_G(X)$ for the weak$^*$-compact convex set of $G$-invariant Borel probability measures on $X$. The action is \emph{uniquely ergodic} if $M_G(X)$ is a singleton.

\subsection{Stone spaces} Recall that a nonempty topological space $X$ is \emph{zero-dimensional} if there exists a basis of clopen sets for its topology. If $X$ is a compact Hausdorff space, this is equivalent to being totally disconnected, i.e.\ every subset with at least two points being disconnected. A \emph{Stone space} is a zero-dimensional compact Hausdorff space. Stone spaces are often called \emph{profinite}, since they are precisely the spaces that are homeomorphic to the projective limit of a directed system of finite discrete spaces. In the present article, we will be interested in actions $G\curvearrowright X$ of countable discrete groups on metrisable Stone spaces.\footnote{Most things also work just fine in the nonmetrisable case with sufficient care (though one will get a nonseparable $C^*$-algebra), but we do not pursue this here.} A Stone space is metrisable if and only if it is homeomorphic to the projective limit over a countable directed set. In summary, for every metrisable Stone space $X$, there exists a sequence $(X_n)_{n\in\Nb}$ of finite sets and surjective connecting maps $\varphi_n : X_{n+1} \to X_n$ such that
\[
    X \cong \varprojlim X_n := \Big\{ (x_n)_{n\in \Nb} \in \prod_{n\in \Nb} X_n : \varphi_n(x_{n+1}) = x_n \text{ for all } n \in \Nb\Big\}.
\]
Suppose that a group $G$ acts on each $X_n$, and that $\varphi_n$ is $G$-equivariant for every $n\in\Nb$. Then the diagonal action $G\curvearrowright \prod_{n\in\Nb} X_n$ given coordinatewise by $g (x_n)_{n\in \Nb} = (gx_n)_{n\in\Nb}$ restricts to a well-defined action on the closed $G$-invariant subset $\varprojlim X_n$.  

\subsection{Profinite completions} A topological group whose underlying space is a Stone space is called a \emph{profinite group}. Equivalently, a profinite group is a projective limit of finite groups, which is metrisable if and only if the index set can be taken to be countable. Let $\sigma = (G_1\ge G_2\ge \cdots)$ be a \emph{chain} in $G$, i.e., a descending sequence of finite-index subgroups. The associated \emph{profinite completion} with respect to $\sigma$ is 
\[
    \widehat{G}_\sigma := \varprojlim G/G_n,
\]
where the connecting maps are given by $gG_{n+1} \mapsto gG_n$. When the chain $\sigma$ is \emph{normal}, i.e., all of the finite-index subgroups $G_n$ are normal, then $\widehat{G}_\sigma$ is a profinite group. In this case, there is a canonical homomorphism $G \to \widehat{G}_\sigma$ given by $g\mapsto (gG_n)_{n\in \Nb}$ whose kernel is $\bigcap_{n\in \Nb} G_n$. In particular, this map is injective if and only if the chain is \emph{separating}, i.e., $\bigcap_{n\in \Nb} G_n = \{e\}$. A countable group $G$ is \emph{residually finite} if for every $g\in G\setminus \{e\}$ there exists a finite-index normal subgroup $N\unlhd G$ such that $g\not\in N$. Equivalently, $G \hookrightarrow \widehat{G}_\sigma$ for some normal chain $\sigma$.  

\subsection{Equicontinuous actions and odometers}

A group action $G\curvearrowright X$ on a metrisable space is \emph{equicontinuous} if the collection of homeomorphisms determined by the action is uniformly equicontinuous, i.e.\ for every $\varepsilon > 0$ there exists $\delta > 0$ such that $d(gx,gy) < \varepsilon$ whenever $d(x,y)\le \delta$, for $g\in G$ and $x,y\in X$.

Equicontinuous actions on Stone spaces admit a number of equivalent characterisations. We refer to \cite[Theorem 2.3]{Pil23} for a proof (for the minimal case, see also \cite{CorPet08,CorMed16,HawSkaWhiZac13,Pil23,CorHau26}).     

\begin{prop}\label{P:equi}
    Let $G\curvearrowright X$ be an action of a countable discrete group on a metrisable Stone space. The following are equivalent:
    \begin{enumerate}
        \item the action is equicontinuous;
        \item there exists a compatible metric on $X$ with respect to which the action is isometric;
        \item there exist finite sets $(X_n)_{n\in \Nb}$, actions $G\curvearrowright X_n$, and $G$-equivariant surjective connecting maps $\varphi_n : X_{n+1} \to X_n$ such that $G\curvearrowright X$ is conjugate to $G\curvearrowright \varprojlim X_n$.
    \end{enumerate}
\end{prop}

We also provide the characterisation of minimal equicontinuous actions, which will be called \emph{subodometers} (see Definition~\ref{D:odometer}). 

\begin{prop}\label{P:min}
    Let $G\curvearrowright X$ be an equicontinuous action of a countable discrete group on a metrisable Stone space. The following are equivalent:
    \begin{enumerate}
        \item the action is minimal;
        \item there exist finite sets $(X_n)_{n\in \Nb}$, transitive actions $G\curvearrowright X_n$ and $G$-equivariant surjective connecting maps $\varphi_n : X_{n+1} \to X_n$ such that $G\curvearrowright X$ is conjugate to $G\curvearrowright \varprojlim X_n$;
        \item there exists a chain $\sigma$ in $G$ such that $G\curvearrowright X$ is conjugate to $G \curvearrowright \widehat{G}_\sigma$.
    \end{enumerate}
\end{prop}

We emphasise that the finite-index subgroups appearing in Proposition~\ref{P:min} may not be normal (when $G$ is not abelian), so that the coset spaces $G/G_n$ may not inherit any group structure. Also, we allow for the possibility that $X$ is finite.  

We closely follow the terminology of \cite{CorPet08,CorMed16,CorHau26}, which also serve as useful supplemental references. 

\begin{df}\label{D:odometer}
    Let $G\curvearrowright X$ be an action of a countable discrete group on a metrisable Stone space. We say that the action is:
    \begin{enumerate}
        \item a \emph{subodometer}, if it is equicontinuous and minimal;
        \item an \emph{odometer}, if it is conjugate to an action $G\curvearrowright \widehat{G}_\sigma$ for some normal chain $\sigma$ in $G$. 
    \end{enumerate}
\end{df}

The relationship between the different actions on Stone spaces considered in this article is summarised in Figure~\ref{fig:hierarchy}. 

\begin{rem}
    By replacing sequences with nets, one may define analogous notions of (sub)odometers for actions on not necessarily metrisable Stone spaces (here, equicontinuity is defined in terms of the unique compatible uniformity on $X$). A group $G$ admits a nonmetrisable subodometer if and only if the collection of finite-index subgroups of $G$ is uncountable. For instance, the universal odometer action of $F_\infty$ will yield a nonmetrisable odometer (see \cite{CorHau26}).   
\end{rem}

% Toggle palette here:
% Use colour-blind friendly version:
\newif\ifHierarchyGreyscale
\HierarchyGreyscalefalse

% Use greyscale version instead:
% \newif\ifHierarchyGreyscale
% \HierarchyGreyscaletrue

\begin{figure}[t]
\centering
\resizebox{0.95\linewidth}{!}{%
\begin{tikzpicture}[
    font=\small,
    title/.style={font=\bfseries, align=center, text=black},
    assumption/.style={
        font=\scriptsize,
        align=left,
        text=black,
        fill=white,
        inner sep=2pt
    },
    ladder/.style={
        -{Latex[length=2mm,width=1.5mm]},
        line width=0.65pt,
        draw=black!75,
        shorten <=2.3pt,
        shorten >=2.3pt
    },
    guide/.style={
        densely dotted,
        draw=black!35,
        line width=0.5pt
    }
]

% ============================================================
% Palette definitions
% Toggle using \HierarchyGreyscaletrue or \HierarchyGreyscalefalse above.
% ============================================================

\ifHierarchyGreyscale

    % Greyscale with light fills
    \tikzset{
        hierarchy outer/.style={
            fill=black!3,
            draw=black!75,
            line width=0.9pt
        },
        hierarchy sub/.style={
            fill=black!6,
            draw=black!80,
            line width=0.9pt
        },
        hierarchy odo/.style={
            fill=black!10,
            draw=black!85,
            line width=0.9pt
        },
        hierarchy free/.style={
            fill=black!14,
            draw=black!90,
            line width=0.9pt
        },
        hierarchy dot outer/.style={fill=black!75},
        hierarchy dot sub/.style={fill=black!80},
        hierarchy dot odo/.style={fill=black!85},
        hierarchy dot free/.style={fill=black!90}
    }

\else

    % Colour-blind friendly Okabe--Ito palette
    \definecolor{HierarchyBlue}{RGB}{0,114,178}
    \definecolor{HierarchyGreen}{RGB}{0,158,115}
    \definecolor{HierarchyOrange}{RGB}{230,159,0}
    \definecolor{HierarchyVermillion}{RGB}{213,94,0}

    \tikzset{
        hierarchy outer/.style={
            fill=HierarchyBlue!8,
            draw=HierarchyBlue!80!black,
            line width=0.9pt
        },
        hierarchy sub/.style={
            fill=HierarchyGreen!9,
            draw=HierarchyGreen!80!black,
            line width=0.9pt
        },
        hierarchy odo/.style={
            fill=HierarchyOrange!13,
            draw=HierarchyOrange!85!black,
            line width=0.9pt
        },
        hierarchy free/.style={
            fill=HierarchyVermillion!10,
            draw=HierarchyVermillion!80!black,
            line width=0.9pt
        },
        hierarchy dot outer/.style={fill=HierarchyBlue!80!black},
        hierarchy dot sub/.style={fill=HierarchyGreen!80!black},
        hierarchy dot odo/.style={fill=HierarchyOrange!85!black},
        hierarchy dot free/.style={fill=HierarchyVermillion!80!black}
    }

\fi

% Nested ellipses, bunched slightly toward the bottom
\draw[hierarchy outer]
    (0,0) ellipse [x radius=5.8, y radius=3.0];

\draw[hierarchy sub]
    (0,-0.55) ellipse [x radius=4.55, y radius=2.15];

\draw[hierarchy odo]
    (0,-1.05) ellipse [x radius=3.25, y radius=1.35];

\draw[hierarchy free]
    (0,-1.55) ellipse [x radius=2.05, y radius=0.65];

% Labels inside ellipses
\node[title] at (0,2.45)
    {Equicontinuous actions};

\node[title] at (0,1.15)
    {Subodometers};

\node[title] at (0,-0.05)
    {Odometers};

\node[title] at (0,-1.55)
    {Free odometers};

% Points on right-hand boundaries of the ellipses,
% horizontally aligned with the corresponding labels
\coordinate (E0) at (3.35,2.45);
\coordinate (E1) at (2.78,1.15);
\coordinate (E2) at (2.18,-0.05);
\coordinate (E3) at (2.05,-1.55);

% Restriction ladder on the right
\coordinate (L0) at (6.20,2.45);
\coordinate (L1) at (6.20,1.15);
\coordinate (L2) at (6.20,-0.05);
\coordinate (L3) at (6.20,-1.55);

% Guide lines from ellipses to ladder
\draw[guide] (E0) -- (L0);
\draw[guide] (E1) -- (L1);
\draw[guide] (E2) -- (L2);
\draw[guide] (E3) -- (L3);

% Arrows and assumptions
\draw[ladder] (L0) -- (L1)
    node[midway, xshift=8pt, anchor=west, assumption]
    {minimality};

\draw[ladder] (L1) -- (L2)
    node[midway, xshift=8pt, anchor=west, assumption]
    {normal chains};

\draw[ladder] (L2) -- (L3)
    node[midway, xshift=8pt, anchor=west, assumption]
    {separating\\normal chains};

% Dots on the ladder
\fill[hierarchy dot outer] (L0) circle (2.3pt);
\fill[hierarchy dot sub] (L1) circle (2.3pt);
\fill[hierarchy dot odo] (L2) circle (2.3pt);
\fill[hierarchy dot free] (L3) circle (2.3pt);

\end{tikzpicture}%
}
\caption{Hierarchy of equicontinuous actions on Stone spaces.}
\label{fig:hierarchy}
\end{figure}
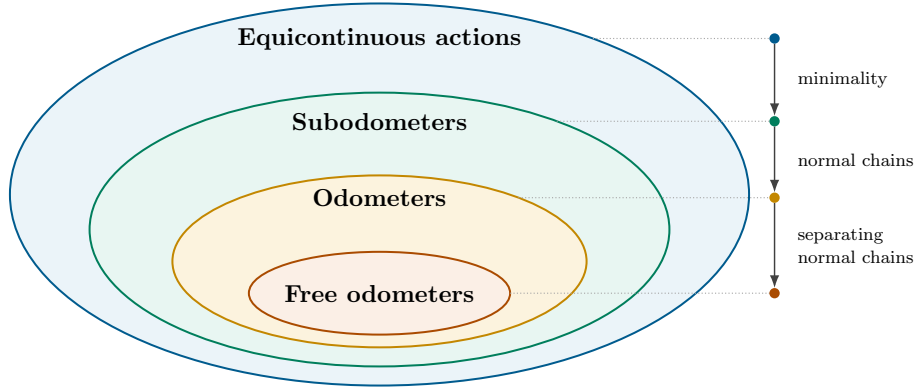

Finally, we summarise some well-known facts about odometers. 

\begin{prop}\label{P:properties}
Let $G$ be a countable discrete group and let $\sigma$ be a normal chain in $G$. Then the odometer action $G\curvearrowright \widehat{G}_\sigma$ is:
\begin{enumerate}
    \item uniquely ergodic;
    \item free if and only if $\sigma$ is separating\footnote{For subodometers, while freeness always implies that the chain is separating, it is known that the converse does not necessarily hold.}; and 
    \item if $G$ is infinite and $\sigma$ is separating, then $\widehat{G}_\sigma$ is
      a Cantor space.\footnote{More generally, an odometer is either finite or a Cantor space. It is finite precisely when the chain is eventually constant, and Cantor precisely when $[G:G_n] \to \infty$.}
\end{enumerate}
\end{prop}

\begin{rem}
    If a countable discrete group admits a free odometer action, then it is necessarily residually finite. 
\end{rem}

\subsection{A decomposition theorem}

We refer the reader to \cite{BroOza08} for details on reduced crossed product $C^*$-algebras.

\begin{prop}\label{P:decomp}
    Let $G\curvearrowright Y$ be an action of a discrete group on a finite set $Y$. Then 
    \[
        C(Y)\rtimes_\lambda G \cong \bigoplus_{[y]\in Y/G} M_{|Gy|}(C^*_\lambda(G_y)).
    \]
    In particular, if the action is transitive then $Y \cong G/H$ for some finite-index $H\le G$, and  
    \[
        C(G/H)\rtimes_\lambda G \cong M_{[G:H]}(C^*_\lambda(H)).
    \]
\end{prop}

\begin{proof}
In the transitive case, fixing any basepoint $y_0 \in Y$, let $H = G_{y_0}$ be the stabiliser subgroup. Then there is a $G$-equivariant bijection $G/H \to Y$ given by $gH \mapsto gy_0$. The isomorphism $C(G/H)\rtimes_\lambda G \cong M_{[G:H]}(C^*_\lambda(H))$ is proved in \cite[Proposition 2.3]{Sca20}. For a general finite action, write $Y$ as the disjoint union of its orbits, $Y=\bigsqcup_{[y]\in Y/G} Gy$. Applying the transitive case to each orbit gives
\[
C(Y)\rtimes_\lambda G \cong \bigoplus_{[y]\in Y/G} C(Gy)\rtimes_\lambda G\cong  \bigoplus_{[y]\in Y/G}
M_{|Gy|}\big(C_\lambda^*(G_y)\big).
\]
This completes the proof.
\end{proof}

\section{Nonnuclear Bunce--Deddens algebras}\label{S:construction}

In this section, we give the definition of (nonnuclear) Bunce--Deddens algebras and record some of their basic properties. 

\begin{df}\label{D:BD}
    Let $G$ be a countably infinite discrete residually finite group and $\sigma = (G_n)_{n\in\Nb}$ a separating normal chain in $G$. Write $\widehat{G}_\sigma := \varprojlim G/G_n$ for the profinite completion of $G$ with respect to $\sigma$. Then the \emph{Bunce--Deddens algebra over $G$ with respect to $\sigma$} is defined as the reduced crossed product 
    \[
        \BD (G,\sigma) := C(\widehat{G}_\sigma)\rtimes_\lambda G. 
    \]
\end{df}

By Proposition~\ref{P:min} and Proposition~\ref{P:properties} (ii), Bunce--Deddens algebras are, equivalently, reduced crossed product $C^*$-algebras associated with free odometer actions on a Cantor space.

\begin{thm}\label{T:BD_properties}
Let $G$ be a countably infinite discrete residually finite group, and let $\sigma$ be a separating normal chain in $G$. Then $\BD (G,\sigma)$ is simple, separable, unital, and admits a unique tracial state. Moreover, $\BD (G,\sigma)$ is nuclear if and only if $G$ is amenable. 
\end{thm}

\begin{proof}
    Unitality and separability are immediate since $\widehat{G}_\sigma$ is compact and metrisable, and $G$ is countable and discrete. Propositions~\ref{P:min} and \ref{P:properties} (ii) imply that the action is minimal and free, respectively, hence $\BD (G,\sigma)$ is simple by \cite{ArcSpi94}. It has a unique tracial state by Proposition~\ref{P:properties} (i), using that for free actions tracial states on $C(\widehat{G}_\sigma)\rtimes_\lambda G$ are in bijective correspondence with invariant measures in $M_G(\widehat{G}_\sigma)$ \cite[Corollary 2.8]{KawTakTom90}. Finally, $\BD (G,\sigma)$ is nuclear if and only if the action $G\curvearrowright \widehat{G}_\sigma$ is amenable, which holds if and only if $G$ is amenable since $M_G(\widehat{G}_\sigma)\ne \emptyset$ (cf.~\cite[Lemma 2.2]{GarGefKraNar23}).   
\end{proof}

\begin{prop}\label{P:limit}
    Let $G\curvearrowright X$ be an equicontinuous action of a countable discrete group on a metrisable Stone space. Then identifying $G\curvearrowright X$ with $G\curvearrowright \varprojlim X_n$ according to Proposition~\ref{P:equi}, we have  
    \[
        C(X)\rtimes_\lambda G \cong \varinjlim \bigoplus_{[y]\in X_n/G} M_{|Gy|}(C^*_\lambda(G_y)).
    \]
    In particular, if $\sigma = (G_n)_{n\in\Nb}$ is a separating normal chain in $G$, then 
    \[
        \BD (G,\sigma) \cong \varinjlim M_{[G:G_n]}(C^*_\lambda(G_n)).
    \]    
\end{prop}

\begin{proof}
    Up to conjugacy, $G\curvearrowright X$ is given by $G\curvearrowright \varprojlim X_n$. Hence
    \[
        C(X)\rtimes_\lambda G \cong C(\varprojlim X_n)\rtimes_\lambda G \cong \varinjlim C(X_n)\rtimes_\lambda G. 
    \]
    Applying Proposition~\ref{P:decomp} yields the desired isomorphism. In the odometer case, the action $G\curvearrowright X_n$ is conjugate to the transitive action $G\curvearrowright G/G_n$, whose stabiliser is $G_n$.
\end{proof}

\section{Regularity in nonnuclear Bunce--Deddens algebras}\label{S:regularity}

In this section, we record the regularity properties of reduced crossed products arising from equicontinuous actions on metrisable Stone spaces. We first briefly recall the relevant definitions.

Let $\Cu(A)$ denote the Cuntz semigroup of a $C^*$-algebra $A$. This is a partially ordered semigroup formed from positive elements of $A\otimes \mathcal{K}$ modulo Cuntz equivalence. We say that $A$ is \emph{pure} (in the sense of Winter \cite{Win12}) if $\Cu(A)$ is almost unperforated and almost divisible. The conjunction of these conditions can be expressed as follows. For all $x,y\in \Cu(A)$ and $n\ge 1$, the inequality $(n+1)x\le ny$ implies $x\le y$, and for all $x\in \Cu(A)$ and $n\ge 1$, there exists $y\in \Cu(A)$ such that $ny\le x\le (n+1)y$ (cf.~\cite[\S~4.2]{AntPerThiVil24}). 

We recall Robert's definition of a \emph{selfless} $C^*$-probability space in the tracial setting relevant to this paper \cite[Definition 2.1]{Rob25}. Let $A$ be a unital $C^*$-algebra and $\tau$ a faithful tracial state on $A$. The pair $(A,\tau)$ is selfless if $A\not\cong \Cb$ and, for some free ultrafilter $\mathcal{U}$, there is a unital trace-preserving embedding of the reduced free product 
\[
	(A,\tau) * (A,\tau) \hookrightarrow (A^\mathcal{U}, \tau^{\mathcal{U}}),
\]	
whose restriction to the first copy of $A$ is the diagonal embedding of $A$ into the ultrapower $A^\mathcal{U}$. When the state $\tau$ is clear from context, we simply say that $A$ is selfless. For example, when $A = C^*_\lambda(H)$ we use the canonical group trace $\tau_H$ and for the subodometer actions in Theorem~\ref{T:self} we use $\mu \circ E$ where $\mu$ is the unique invariant probability measure and $E$ is the canonical conditional expectation onto $C(X)$.  

\begin{thm}\label{T:sr1pure}
    Let $G\curvearrowright X$ be an equicontinuous action of a countable discrete group on a metrisable Stone space. Suppose that for all finite-index subgroups $H\le G$, the reduced group $C^*$-algebra $C^*_\lambda(H)$ is pure (has stable rank one). Then the reduced crossed product $C(X)\rtimes_\lambda G$ is pure (has stable rank one). 
\end{thm}

\begin{proof}
    By Proposition~\ref{P:limit}, $C(X)\rtimes_\lambda G$ is isomorphic to the inductive limit of a sequence of the form $\bigoplus_{j=1}^{N_n} M_{[G:H_{n,j}]}(C^*_\lambda(H_{n,j}))$ for some $N_n \in \Nb$ and finite-index subgroups $H_{n,j}\le G$. Pureness is preserved by taking matrix amplifications, direct sums and inductive limits; see Theorems 4.11 and 3.8 in \cite{PerThiVil25}.\footnote{One way to see that pureness is invariant under matrix amplification is to note that pureness is a property of the Cuntz semigroup; in particular, invariant under stabilisation.} Stable rank one has the same permanence properties by Theorems 3.3, 5.2, and 5.1 in \cite{Rie83}. Thus the result follows by our assumption. 
\end{proof}

\begin{cor}\label{C:sr1pure}
    Let $G$ be a countable discrete group. Suppose that $G$ is acylindrically hyperbolic or has Ozawa's property PHP. Let $G\curvearrowright X$ be an equicontinuous action on a metrisable Stone space. Then the reduced crossed product $C(X)\rtimes_\lambda G$ is pure and has stable rank one. 
\end{cor}

\begin{proof}
    By Theorem~\ref{T:sr1pure}, it suffices to show that $C^*_\lambda(H)$ is pure and has stable rank one for every finite-index subgroup $H\le G$. If $G$ is acylindrically hyperbolic, then $H$ is acylindrically hyperbolic by \cite[Lemma 3.9]{MinOsi15} (see also \cite[Lemma 1]{MinOsi19}), so $C^*_\lambda(H)$ is pure and has stable rank one by \cite[Corollary~G(i)]{FloKliCobPag26} (see also \cite{AmrGaoKunPat25,Rau25,RauThiVil25}). If $G$ has Ozawa's property PHP, then $H$ has property PHP by \cite[Proposition~2.3]{Flo26}. Hence $C^*_\lambda(H)$ is selfless by \cite[Theorem~14]{Oza25}. It follows from \cite[Theorem~3.1]{Rob25} and \cite[Theorem~5.13]{AntPerThiVil24} that $C^*_\lambda(H)$ is pure and has stable rank one. 
\end{proof}

The acylindrically hyperbolic case of Corollary~\ref{C:sr1pure} includes groups with nontrivial finite radical (hence not $C^*$-simple). Groups with property PHP include those admitting a topologically free extreme boundary action, in particular acylindrically hyperbolic groups with trivial finite radical \cite[Proposition~15]{Oza25}. They also include nontrivial countable linear groups with trivial amenable radical, by the work of Vigdorovich \cite{Vig26}. Some concrete examples of groups covered by Corollary~\ref{C:sr1pure} are the following. 

\begin{enumerate}
    \item All free products $G * H$ of nontrivial groups with $|H| > 2$. 
    \item Surface groups, i.e., fundamental groups $\pi_1(\Sigma_g)$ of closed orientable surfaces $\Sigma_g$ of genus $g\ge 2$. 
    \item Mapping class groups $\operatorname{MCG}(\Sigma_{g,p})$ of a closed surface of genus $g$ with $p$ punctures, where $g \ne 0$ and $p > 3$. 
    \item Baumslag--Solitar groups $\operatorname{BS}(m,n)$ with $|m|\ne |n|$ and $|m|,|n| > 1$. 
    \item Zariski-dense subgroups of $\operatorname{PSL}(d,\Rb)$ with $d\ge 2$.
\end{enumerate}

The list of groups in Corollary~\ref{C:sr1pure} is far from exhaustive. For instance, there are groups $G$ for which $C^*_\lambda(G)$ has stable rank one, while it is not currently known whether $C^*_\lambda(G)$ is pure (see, e.g.\ the recent work \cite{KerPet26}). 

For subodometer actions, the conclusion of Theorem~\ref{T:sr1pure} can often be strengthened. 

\begin{thm}\label{T:self}
    Let $G\curvearrowright X$ be a subodometer action of a countable discrete group on a metrisable Stone space. Suppose that $C^*_\lambda(H)$ is selfless for all finite-index subgroups $H\le G$. Then $C(X)\rtimes_\lambda G$ is selfless.
\end{thm}

\begin{proof}
    By Propositions~\ref{P:min} and \ref{P:limit}, $C(X)\rtimes_\lambda G$ is the inductive limit of a sequence $(A_n)_n$ with unital connecting maps, where
    \[
    		A_n \cong M_{[G:G_n]}(C^*_\lambda(G_n))
    \]
for some finite-index subgroups $G_n \le G$. By hypothesis and \cite[Theorem 4.3]{Rob25}, $A_n$ is selfless with respect to the product trace $\tau_n = \mathrm{tr}_{[G:G_n]} \otimes \tau_{G_n}$, where $\mathrm{tr}$ denotes the normalised matrix trace. By \cite[Theorem 3.1]{Rob25}, $\tau_n$ is the unique tracial state on $A_n$. Since $\mu \circ E$ is a tracial state on $C(X)\rtimes_\lambda G$ and the connecting maps are unital, its restriction to $A_n$ must therefore equal $\tau_n$. The conclusion now follows from \cite[Theorem 4.1]{Rob25}.  
\end{proof}

\begin{cor}\label{C:self}
    Let $G$ be a countable discrete group with Ozawa's property PHP, and let $G\curvearrowright X$ be a subodometer action on a metrisable Stone space. Then the reduced crossed product $C(X)\rtimes_\lambda G$ is selfless. 
\end{cor}

\begin{proof}
    By \cite[Proposition~2.3]{Flo26}, every finite-index subgroup $H \le G$ has property PHP. Thus $C^*_\lambda(H)$ is selfless by \cite[Theorem~14]{Oza25}. The conclusion follows from Theorem~\ref{T:self}.
\end{proof}

The proofs of Theorems~\ref{T:sr1pure} and \ref{T:self} raise the following question. 

\begin{question}
    Let $G$ be a countable discrete group and $H\le G$ a finite-index subgroup. If $C^*_\lambda(G)$ is selfless, pure, or has stable rank one, does the same hold for $C^*_\lambda(H)$? 
\end{question}

\begin{rem}
    Stable rank one has previously been established for certain simple crossed products of actions of free product groups in work by Geffen, Kerr and the author via the method of square divisibility \cite{BelGefKer25}. Corollary~\ref{C:sr1pure} provides a more direct proof that the universal odometer action of $F_2$ yields a crossed product with stable rank one (cf.~\cite[Remark 9.4]{BelGefKer25}). We also point out that Corollary~\ref{C:sr1pure} encompasses actions which are not (weakly) squarely divisible, so the results of \cite{BelGefKer25} do not apply; see Example~\ref{E:con_sub}.
\end{rem}

\begin{example}\label{E:con_sub}
    By Selberg's theorem, $\operatorname{SL}(2,\Zb)$ has property $(\tau)$ with respect to the family of congruence subgroups $\Gamma(n) = \ker(\operatorname{SL}(2,\Zb) \to \operatorname{SL}(2,\Zb/n\Zb))$ (see \cite{Sel65,Lub05}). Identify $F_2$ with the subgroup of $\operatorname{SL}(2,\Zb)$ generated by 
    \[
    		\begin{pmatrix}
    		1 & 2\\
    		0 & 1
    		\end{pmatrix}, \quad \begin{pmatrix}
    		1 & 0\\
    		2 & 1
    		\end{pmatrix}.
    \]
It is well-known that these matrices freely generate a subgroup of index $12$. Applying \cite[Theorem 6]{AbeEle12} to the chain $(\Gamma(n!))_{n\in\Nb}$, we see that $F_2$ has property $(\tau)$ with respect to the separating normal chain $\sigma = (\Gamma(n!)\cap F_2)_{n\in\Nb}$. This implies that the odometer action $F_2\curvearrowright (\widehat{F_2})_\sigma$ has spectral gap with respect to normalised Haar measure and is therefore strongly ergodic. By \cite[Proposition 3.15]{BelGefKer25}, the action is not weakly squarely divisible, but its crossed product nevertheless has stable rank one by Corollary~\ref{C:sr1pure}.
\end{example}

R{\o}rdam proved in \cite{Ror04} that a simple unital $\mathcal{Z}$-stable $C^*$-algebra is pure, and the Toms--Winter conjecture -- which has been proved in the monotracial case -- predicts that the converse holds among infinite-dimensional nuclear $C^*$-algebras. For nonnuclear $C^*$-algebras these properties often diverge, as demonstrated by the following. 

\begin{prop}\label{P:Zstable}
    Let $G$ be a countable discrete residually finite group that is not inner amenable. Then for every separating normal chain $\sigma$ in $G$, the Bunce--Deddens algebra $\BD (G,\sigma)$ is not $\mathcal{Z}$-stable. 
\end{prop}

\begin{proof}
    Suppose for the sake of contradiction that $A:= \BD (G,\sigma)$ is $\mathcal{Z}$-stable. Then the von Neumann algebra $\pi_\tau(A)''\cong L^\infty(\widehat{G}_\sigma)\rtimes G$ generated by the GNS representation of the unique tracial state $\tau$ on $A$ is McDuff (see, e.g., \cite[Proposition 5.17]{CarCasEviGabSchTikWhi24}). This implies that $G$ is McDuff in the sense of \cite{DepVae18} and hence inner amenable \cite[Proposition 4.1]{DepVae18}. We arrive at the desired contradiction. 
\end{proof}

\begin{rem}\label{R:inn_amen}
    Well-known examples of groups that are not inner amenable, hence covered by Proposition~\ref{P:Zstable}, include all free products $G * H$ of nontrivial groups with $|H| > 2$ \cite{Eff75}, ICC property (T) groups \cite[Theorem 11]{AkeWal81}, and ICC acylindrically hyperbolic groups \cite[Theorem 2.35]{DahGuiOsi17}. 
\end{rem}

We now record our real rank zero result (see \cite{BroPed91} for definitions).

\begin{prop}\label{P:rr0}
    Let $G$ be a countably infinite discrete residually finite group and let $\sigma$ be a separating normal chain in $G$, and write $A:= \BD(G,\sigma)$. Writing $\tau$ for the unique tracial state on $A$, and $m_n = [G:G_n]$, we have
    \begin{equation}\label{E:trace_inclusion}
    \bigcup_{n\in \Nb} \frac{1}{m_n} \Zb \subseteq \tau_*(K_0(A)).
    \end{equation}
	In particular, $\tau_*(K_0(A))$ is dense in $\Rb$. If, moreover, $A$ is selfless, then $A$ has real rank zero. 
\end{prop}

\begin{proof}
      Let $\pi_n : \widehat{G}_\sigma \to G/G_n$ be the quotient map and let $p_n \in C(\widehat{G}_\sigma)\subseteq A$ be the characteristic function of $\pi_n^{-1}(\{G_n\})$. For every $n\in \Nb$, the $m_n$ distinct translates of $p_n$ are mutually orthogonal projections which sum to $1$. They are unitarily equivalent in $A$ via the canonical unitaries implementing the action of $G$. Hence $m_n[p_n]_0 = [1]_0$ and $\tau(p_n) = 1/m_n$. Thus (\ref{E:trace_inclusion}) holds. Since the chain is separating and $G$ is infinite, the indices $m_n$ tend to infinity, and therefore $\tau_*(K_0(A))$ contains arbitrarily small positive elements and so is dense in $\Rb$.  
      
   Now suppose that $A$ is selfless. By Theorem~\ref{T:BD_properties}, $A$ is simple and unital.  By \cite[Theorem 3.1]{Rob25}, $A$ also has stable rank one and strict comparison, and $\tau$ is its unique normalised $2$-quasitrace. In particular, $\Cu(A)$ is almost unperforated. Since $T(A) = \{\tau\}$, evaluation at $\tau$ identifies the canonical pairing $\rho : K_0(A) \to \operatorname{Aff}(T(A))$ with $\tau_*$, whose image is dense by the first part of the proof. In the proof of \cite[Proposition 7.1]{Ror04}, exactness is used only to ensure that every normalised $2$-quasitrace is a trace. Here this follows from the uniqueness of the normalised $2$-quasitrace, so the same proof gives that $A$ has real rank zero. 
\end{proof}

\begin{cor}\label{C:rr0}
    Fix $d\ge 2$. Let $\sigma$ be a separating normal chain in $F_d$. Then $\BD(F_d,\sigma)$ has real rank zero. 
\end{cor}

\begin{proof}
    By Corollary~\ref{C:self} $\BD(F_d,\sigma)$ is selfless and so Proposition~\ref{P:rr0} now applies. 
\end{proof}

We at last have everything we need to prove Theorem~\ref{T:A}. 

\begin{proof}[Proof of Theorem~\ref{T:A}]
    By Theorem~\ref{T:BD_properties}, $A:=\BD (F_d,\sigma)$ is simple, separable, unital and admits a unique tracial state; since $F_d$ is nonamenable, it is also nonnuclear.  Since $F_d$ is not inner amenable, Proposition~\ref{P:Zstable} and Remark~\ref{R:inn_amen} show that $A$ is not $\mathcal{Z}$-stable, and real rank zero is proved in Corollary~\ref{C:rr0}. Corollary~\ref{C:self} proves that $A$ is selfless, hence pure and with stable rank one; alternatively, Corollary~\ref{C:sr1pure} proves that $A$ is pure and has stable rank one.
\end{proof}

\section{$K$-theory for Bunce--Deddens algebras}\label{S:Kthy}

In this section we compute the $K$-theory of $\BD(F_d,\sigma)$, where $F_d$ is a finitely generated free group and $\sigma = (G_n)_{n\in\Nb}$ is a separating normal chain. A standard reference on $K$-theory is \cite{RorLarLau00}. 

\begin{notation}\label{N:chain}
For a separating normal chain $\sigma = (G_n)_{n\in \Nb}$ in $F_d$, put $m_n = [F_d: G_n]$ and $q_n = [G_n : G_{n+1}] = m_{n+1} / m_n$. Define 
\[
    \Qb_\sigma = \bigcup_{n\in \Nb} \frac{1}{m_n}\Zb \subseteq \Qb. 
\]
\end{notation}

\begin{lemma}\label{L:fin_stage}
    Fix an integer $d\ge 2$. Let $G_n \le F_d$ be a finite-index subgroup, and write $m_n := [F_d:G_n]$. Then  
    \[
        K_0(C(F_d/G_n)\rtimes_\lambda F_d) \cong \Zb \quad \text{and} \quad K_1(C(F_d/G_n)\rtimes_\lambda F_d) \cong \Zb^{1+m_n(d-1)}.  
    \]
    Under the identification $K_0(C(F_d/G_n)\rtimes_\lambda F_d) \cong K_0(C^*_\lambda(G_n)) \cong \Zb$, the unit in $C(F_d/G_n)\rtimes_\lambda F_d$ corresponds to $m_n \in \Zb$.
\end{lemma}

\begin{proof}
    By the Nielsen--Schreier formula, we have $G_n \cong F_{1 + m_n(d-1)}$. By \cite[Corollary 3.2]{PimVoi82}, $K_0(C^*_\lambda(F_r)) \cong \Zb[1]_0$ and $K_1(C^*_\lambda(F_r)) \cong \Zb^r$, so, since $K$-theory is invariant under matrix amplification, the result follows from Proposition~\ref{P:decomp}. The unit of $M_{m_n}(C^*_\lambda(G_n))$ represents $m_n[1_{C^*_\lambda(G_n)}]_0$, which corresponds to $m_n \in \Zb$ under the isomorphism $K_0(C^*_\lambda(G_n))\cong \Zb$.  
\end{proof}

\begin{thm}[Theorem~\ref{T:B}]\label{T:Kthy}
    Fix $d\ge 2$. Let $\sigma = (G_n)_{n\in\Nb}$ be a separating normal chain in $F_d$. Then, writing $A = \BD(F_d,\sigma)$ and using Notation~\ref{N:chain},
    \[
        (K_0(A),K_0(A)^+, [1]_0) \cong (\Qb_\sigma,\Qb_\sigma \cap [0,\infty), 1), \quad K_1(A) \cong \Zb^{(\Nb)}.
    \]
    Under the first identification, the trace pairing $\tau_* : K_0(A) \to \Rb$ induced by the unique tracial state $\tau$ on $A$ is the inclusion $\Qb_\sigma \hookrightarrow \Rb$. 
\end{thm}

\begin{proof}
    Let $X = (\widehat{F_d})_\sigma$ and $X_n = F_d/G_n$. By Proposition~\ref{P:limit}, $A$ decomposes into an inductive limit $A \cong \varinjlim C(X_n)\rtimes_\lambda F_d$ where the connecting maps are induced by pullback along the quotient maps $X_{n+1} \to X_n$. By continuity of $K$-theory, we have, for $i\in \{0,1\}$, 
    \begin{equation}\label{E:cty}
        K_i(A) \cong \varinjlim K_i(C(X_n)\rtimes_\lambda F_d).
    \end{equation}
    Lemma~\ref{L:fin_stage} gives $K_0(C(X_n)\rtimes_\lambda F_d) \cong \Zb$. Since the connecting maps are unital, the induced map on $K_0$ sends $m_n$ to $m_{n+1}$ under the identifications of Lemma~\ref{L:fin_stage}. As this is a homomorphism $\Zb \to \Zb$, it is therefore multiplication by $m_{n+1}/m_n = q_n$. Therefore, 
    \[
        K_0(A) \cong \varinjlim (\Zb \stackrel{\times q_1}{\longrightarrow} \Zb \stackrel{\times q_2}{\longrightarrow} \Zb \longrightarrow \cdots) \cong \Qb_\sigma,
    \]
    where the copy of $\Zb$ at stage $n$ is identified with $m_n^{-1}\Zb \subseteq \Qb$. Since the unit $[1]_0$ corresponds at each stage with $m_n \in \Zb$, it is mapped under this identification to $1 \in \Qb$. 
   
   Let $\tau$ be the unique tracial state on $A$, and let $p_n \in C(X)$ denote the characteristic function of $\pi_n^{-1}(\{G_n\})$, where $\pi_n : X \to X_n$ is the projection. As in the proof of Proposition~\ref{P:rr0}, we have $m_n[p_n]_0 = [1]_0$ and $\tau(p_n) = 1/m_n$. The first equality identifies $[p_n]_0$ with $1/m_n$ in $\Qb_\sigma$ and the second shows that $\tau_*$ is the inclusion $\Qb_\sigma \hookrightarrow \Rb$, since the elements $1/m_n$ generate $\Qb_\sigma$. Positivity of $\tau_*$ implies that $K_0(A)^+ \subseteq \Qb_\sigma \cap [0,\infty)$. Conversely, every element of $\Qb_\sigma \cap [0,\infty)$ has the form $k/m_n$ for some integer $k \ge 0$ and therefore corresponds to the positive class $k[p_n]_0$.  
    
  For the $K_1$-group computation, fix free generators $s_1,\ldots,s_d$ of $F_d$. For $Y \in \{X,X_n\}$, define $\partial_Y : C(Y,\Zb)^d \to C(Y,\Zb)$ by 
  \[
  	\partial_Y(f_1,\ldots,f_d) = \sum_{j=1}^d (f_j - f_j \circ s_j^{-1}). 
  \]
  Since $Y$ is zero-dimensional, $K_0(C(Y)) \cong C(Y,\Zb)$ and $K_1(C(Y)) = 0$. The Pimsner--Voiculescu exact sequence for free groups \cite{PimVoi82} therefore yields an isomorphism $K_1(C(Y)\rtimes_\lambda F_d) \cong \ker \partial_Y$. By naturality of the exact sequence, the connecting maps are identified with componentwise pullback along $X_{n+1} \to X_n$. Since every continuous integer-valued function on $X$ factors through some $X_n$, $C(X,\Zb)$ is the increasing union of the groups $C(X_n,\Zb)$ under pullback. Each of these groups is finitely generated and free abelian. Moreover, each map $X_{n+1} \to X_n$ admits a set-theoretic section, so the pullback $C(X_{n},\Zb) \to C(X_{n+1},\Zb)$ is split injective as a homomorphism of abelian groups. Thus $C(X,\Zb)$ is a countable direct sum of finitely generated free abelian groups. It follows that $\ker \partial_X$ is countable and free abelian. Finally, Lemma~\ref{L:fin_stage}  gives 
  \[
  	\ker\partial_{X_n} \cong K_1(C(X_n)\rtimes_\lambda F_d) \cong \Zb^{1+m_n(d-1)}. 
  \] 
  Since the projection $X \to X_n$ is surjective and equivariant, pullback embeds $\ker\partial_{X_n}$ in $\ker\partial_X$. Since the chain is separating, $m_n \to \infty$, and since $d\ge 2$, these subgroups have unbounded rank. Hence $\ker\partial_X$ is free abelian of countably infinite rank which yields $K_1(A) \cong \Zb^{(\Nb)}$. 
\end{proof}

\begin{rem}
    Using the Baum--Connes conjecture, we expect one can carry out many more computations of $K$-theory for nonnuclear Bunce--Deddens algebras, along the lines of \cite{Dee23}. 
\end{rem}

\begin{cor}[Corollary~\ref{C:C}]
    For each $d\ge 2$, there exist continuum many pairwise nonisomorphic Bunce--Deddens algebras over $F_d$.  
\end{cor}

\begin{proof}
    By Theorem~\ref{T:Kthy}, it suffices to construct continuum many separating normal chains $\sigma$ in $F_d$ for which the groups $\Qb_\sigma$ are pairwise nonisomorphic. Let $\P$ denote the set of prime numbers. Since $F_d$ is residually $2$-finite, there exists a separating normal chain $R_1\ge R_2\ge \cdots$ in $F_d$ such that $[F_d : R_n]$ is a power of $2$ for all $n\in \Nb$. Now let $S\subseteq \P$ be any subset with $2\in S$. Choose an increasing sequence of finite subsets $E_1\subseteq E_2 \subseteq \cdots \subseteq S$ with $2\in E_n$ for all $n$ and $\bigcup_{n\in\Nb} E_n = S$ and define $N_n = (\prod_{q\in E_n} q)^n$. Then $N_n \mid N_{n+1}$, the prime divisors of each $N_n$ lie in $S$, the primes appearing among the $N_n$'s are precisely the primes in $S$, and every prime in $S$ occurs with arbitrarily large exponent. Choose an epimorphism $\chi : F_d \to \Zb$ and set $L_n = \chi^{-1}(N_n\Zb)$. Then each $L_n\le F_d$ is normal, $[F_d : L_n] = N_n$ and $L_{n+1} \le L_n$ for all $n\in \Nb$, since $N_n\mid N_{n+1}$. Taking $G_n := R_n \cap L_n$ defines a normal chain $\sigma_S$ in $F_d$ that is separating since $\bigcap_{n\in \Nb} G_n \subseteq \bigcap_{n\in \Nb} R_n = \{e\}$. Since $G_n\subseteq L_n$, we have $N_n = [F_d : L_n] \mid [F_d : G_n] =: m_n$. Hence every prime in $S$ appears in the sequence $(m_n)_{n\in \Nb}$ with arbitrarily large exponent. Conversely, the map $F_d/(R_n\cap L_n) \to F_d/R_n \times F_d/L_n$ is injective, so $m_n$ divides $[F_d : R_n][F_d:L_n]$. The first factor $[F_d : R_n]$ is a power of $2$ by construction and $[F_d : L_n] = N_n$ has prime divisors in $S$. Since $2\in S$, no prime $p \in \P\setminus S$ divides any $m_n$. It follows that $\Qb_{\sigma_S} = \Zb[1/q : q\in S]$. For each prime $p$, multiplication by $p$ on $\Qb_{\sigma_S}$ is surjective if and only if $p \in S$. As group isomorphisms preserve this property, distinct choices of $S$ give nonisomorphic $K_0$-groups, and hence nonisomorphic Bunce--Deddens algebras. Since there exists a continuum of subsets of $\P$ containing $2$, this completes the proof.
\end{proof}

\begin{rem}
	In a previous version of this paper, we raised the question -- originally due to Thiel -- of whether $C^*_\lambda(G)$ is pure if (and only if) $G$ is nonamenable. This question has since been resolved negatively in \cite{Bel26}. This means that a new method is required to establish pureness of every nonnuclear Bunce--Deddens algebra.
\end{rem}


\begin{thebibliography}{999}

\bibitem{AbeEle12}
M. Ab\'ert and G. Elek.
Dynamical properties of profinite actions.
{\it Ergodic Theory Dynam. Systems} {\bf 32} (2012), 1805--1835.

\bibitem{AkeWal81}
C.~A. Akemann and M. Walter.
Unbounded negative definite functions. 
{\it Canad. J. Math.} {\bf 33} (1981), 862--871.

\bibitem{AmrGaoKunPat25}
T. Amrutam, D. Gao, S. Kunnawalkam Elayavalli, and G. Patchell. 
Strict comparison in reduced group $C^*$-algebras.
{\it Invent. Math.} {\bf 242} (2025), 639--657. 

\bibitem{AntPerThiVil24}
R. Antoine, F. Perera, H. Thiel, and E. Vilalta.
Pure $C^*$-algebras.
arXiv:2406.11052. 

\bibitem{ArcSpi94}
R.~J. Archbold and J.~S. Spielberg.
Topologically free actions and ideals in discrete $C^*$-dynamical systems.
{\it Proc. Edinburgh Math. Soc. (2)} {\bf 37} (1994), 119--124. 

\bibitem{BelGefKer25}
J. Bell, S. Geffen, and D. Kerr.
Stable rank one in nonnuclear crossed products.
arXiv:2511.20132.

\bibitem{Bel26}
J. Bell.
Nonamenable groups whose reduced group $C^*$-algebras are not pure.
arXiv:2609.10653v2.

\bibitem{BroPed91}
L.~G. Brown and G.~K. Pedersen. 
$C^*$-algebras of real rank zero.
{\it J. Funct. Anal.} {\bf 99} (1991), 131--149.

\bibitem{BroOza08}
N.~P. Brown and N. Ozawa.
{\it $C^*$-Algebras and Finite-Dimensional Approximations}. 
Graduate Studies in Mathematics {\bf 88}. American Mathematical Society, Providence, RI, 2008.

\bibitem{BunDed75}
J.~W. Bunce and J.~A. Deddens.
A family of simple $C^*$-algebras related to weighted shift operators.
{\it J. Funct. Anal.} {\bf 19} (1975), 13--24.

\bibitem{CarCasEviGabSchTikWhi24}
J.~R. Carri\'on, J. Castillejos, S. Evington, J. Gabe, C. Schafhauser, A. Tikuisis, and S. White.
Tracially complete $C^*$-algebras.
To appear in {\it Mem. Amer. Math. Soc.}

\bibitem{Car11}
J.~R. Carri\'on.
Classification of a class of crossed product $C^*$-algebras associated with residually finite groups.
{\it J. Funct. Anal.} {\bf 260} (2011), 2815--2825.

\bibitem{CorHau26}
M.~I. Cortez and T. Hauser.
Minimal equicontinuous actions on Stone spaces.
arXiv:2602.05756.

\bibitem{CorMed16}
M.~I. Cortez and K. Medynets.
Orbit equivalence rigidity of equicontinuous systems.
{\it J. Lond. Math. Soc. (2)} {\bf 94} (2016), 545--556.

\bibitem{CorPet08}
M.~I. Cortez and S. Petite.
$G$-odometers and their almost 1-1 extensions.
{\it J. Lond. Math. Soc. (2)} {\bf 78} (2008), 1--20. 

\bibitem{DahGuiOsi17}
F. Dahmani, V. Guirardel, and D. Osin.
Hyperbolically embedded subgroups and rotating families in groups acting on hyperbolic spaces.
{\it Mem. Amer. Math. Soc.} {\bf 245} (2017), v+152.

\bibitem{Dee23}
R.~J. Deeley.
A counterexample to the HK-conjecture that is principal. 
{\it Ergodic Theory Dynam. Systems} {\bf 43} (2023), 1829--1846.

\bibitem{DepVae18}
T. Deprez and S. Vaes.
Inner amenability, property Gamma, McDuff $\mathrm{II}_1$ factors and stable equivalence relations. 
{\it Ergodic Theory Dynam. Systems} {\bf 38} (2018), 2618--2624.

\bibitem{Eff75}
E.~G. Effros.
Property ${\Gamma}$ and inner amenability.
{\it Proc. Amer. Math. Soc.} {\bf 47} (1975), 483--486.

\bibitem{Ell76}
G.~A. Elliott.
On the classification of inductive limits of sequences of semisimple finite-dimensional algebras.
{\it J. Algebra} {\bf 38} (1976), 29--44. 

\bibitem{Ell93}
G.~A. Elliott.
On the classification of $C^*$-algebras of real rank zero.
{\it J. Reine Angew. Math.} {\bf 443} (1993), 179--219.

\bibitem{EllGon96}
G.~A. Elliott and G. Gong.
On the classification of $C^*$-algebras of real rank zero. II.
{\it Ann. of Math. (2)} {\bf 144} (1996), 497--610. 

\bibitem{EviSib26}
S. Evington and P. Sibbel.
$C^*$-diagonals with Cantor spectrum in Cuntz algebras.
{\it J. Funct. Anal.} {\bf 290} (2026), paper no. 111418. 
 
\bibitem{FloKliCobPag26}
F. Flores, M. Klisse, M. \'{O} Cobhthaigh, and M. Pagliero.
Pureness and stable rank one for reduced twisted group $C^*$-algebras of certain group extensions. 
arXiv:2601.19758v3.

\bibitem{Flo26}
F. Flores.
PHP decompositions and primitivity of group rings for linear groups.
arXiv:2609.14814.

\bibitem{GarGefGesKopNar24}
E. Gardella, S. Geffen, R. Gesing, G. Kopsacheilis, and P. Naryshkin.
Essential freeness, allostery, and $\mathcal{Z}$-stability of crossed products.
arXiv:2405.04343.

\bibitem{GarGefKraNar23}
E. Gardella, S. Geffen, J. Kranz, and P. Naryshkin.
Classifiability of crossed products by nonamenable groups.
{\it J. Reine Angew. Math.} {\bf 797} (2023), 285--312.

\bibitem{GioPutSka19}
T. Giordano, I.~F. Putnam, and C.~F. Skau.
$\Zb^d$-odometers and cohomology.
{\it Groups Geom. Dyn.} {\bf 13} (2019), 909--938.

\bibitem{Gli60}
J.~G. Glimm.
On a certain class of $C^*$-algebras.
{\it Trans. Amer. Math. Soc.} {\bf 95} (1960), 318--340. 

\bibitem{HawSkaWhiZac13}
A. Hawkins, A. Skalski, S. White, and J. Zacharias.
On spectral triples on crossed products arising from equicontinuous actions.
{\it Math. Scand.} {\bf 113} (2013), 262--291. 

\bibitem{Jos24b}
M. Joseph.
Amenable wreath products with non almost finite actions of mean dimension zero.
{\it Trans. Amer. Math. Soc.} {\bf 377} (2024), 1321--1333.

\bibitem{KawTakTom90}
S. Kawamura, H. Takemoto, and J. Tomiyama. 
State extensions in transformation group $C^*$-algebras.
{\it Acta Sci. Math. (Szeged)} {\bf 54} (1990), 191--200. 

\bibitem{KerPet26}
D. Kerr and S. Petrakos.
Topological full groups and stable rank one.
arXiv:2607.04300.

\bibitem{Lub05}
A. Lubotzky.
What is ... property $(\tau)$?
{\it Notices Amer. Math. Soc.} {\bf 52} (2005), 626--627.

\bibitem{MinOsi15}
A. Minasyan and D. Osin.
Acylindrical hyperbolicity of groups acting on trees.
{\it Math. Ann.} {\bf 362} (2015), 1055--1105.

\bibitem{MinOsi19}
A. Minasyan and D. Osin.
Correction to: Acylindrical hyperbolicity of groups acting on trees.
{\it Math. Ann.} {\bf 373} (2019), 895--900.

\bibitem{Orf10}
S. Orfanos.
Generalized Bunce--Deddens algebras.
{\it Proc. Amer. Math. Soc.} {\bf 138} (2010), 299--308. 

\bibitem{Oza25}
N. Ozawa.
Proximality and selflessness for group $C^*$-algebras.
arXiv:2508.07938.

\bibitem{PerThiVil25}
F. Perera, H. Thiel, and E. Vilalta.
Extensions of pure $C^*$-algebras.
To appear in {\it Trans. Amer. Math. Soc.}  

\bibitem{Phi05}
N.~C. Phillips.
Crossed products of the Cantor set by free minimal actions of $\Zb^d$.
{\it Comm. Math. Phys.} {\bf 256} (2005), 1--42.

\bibitem{Pil23}
S.~J. Pilgrim.
Isometric actions and finite approximations.
{\it Ergodic Theory Dynam. Systems} {\bf 43} (2023), 2464--2470. 

\bibitem{PimVoi80}
M.~V. Pimsner and D. Voiculescu.
Exact sequences for $K$-groups and Ext-groups of certain cross-product $C^*$-algebras.
{\it J. Operator Theory} {\bf 4} (1980), 93--118.

\bibitem{PimVoi82}
M.~V. Pimsner and D. Voiculescu.
$K$-groups of reduced crossed products by free groups.
{\it J. Operator Theory} {\bf 8} (1982), 131--156.

\bibitem{Put89}
I.~F. Putnam.
The $C^*$-algebras associated with minimal homeomorphisms of the Cantor set.
{\it Pacific J. Math.} {\bf 136} (1989), 329--353.

\bibitem{RauThiVil25}
S. Raum, H. Thiel, and E. Vilalta.
Strict comparison for twisted group $C^*$-algebras.
arXiv:2505.18569. 

\bibitem{Rau25}
S. Raum. 
Twisted group $C^*$-algebras of acylindrically hyperbolic groups have stable rank one. 
To appear in {\it Groups Geom. Dyn.}

\bibitem{Rie81}
M.~A. Rieffel.
$C^*$-algebras associated with irrational rotations.
{\it Pacific J. Math.} {\bf 93} (1981), 415--429. 

\bibitem{Rie83}
M.~A. Rieffel.
Dimension and stable rank in the $K$-theory of $C^*$-algebras.
{\it Proc. Lond. Math. Soc. (3)} {\bf 46} (1983), 301--333.

\bibitem{Rob25}
L. Robert.
Selfless $C^*$-algebras.
{\it Adv. Math.} {\bf 478} (2025), paper no. 110409.

\bibitem{RorLarLau00}
M. Rørdam, F. Larsen, and N.~J. Laustsen.
{\it An Introduction to $K$-Theory for $C^*$-Algebras}.
London Mathematical Society Student Texts {\bf 49}.
Cambridge University Press, Cambridge, 2000.

\bibitem{Ror04}
M. R{\o}rdam.
The stable and the real rank of $\mathcal{Z}$-absorbing $C^*$-algebras.
{\it Internat. J. Math.} {\bf 15} (2004), 1065--1084.

\bibitem{Sca20}
E. Scarparo.
Homology of odometers.
{\it Ergodic Theory Dynam. Systems} {\bf 40} (2020), 2541--2551.

\bibitem{Sel65}
A. Selberg.
On the estimation of Fourier coefficients of modular forms.
In {\it Proceedings of Symposia in Pure Mathematics, Vol. VIII}, 1--15, Amer. Math. Soc., Providence, RI, 1965.

\bibitem{Vig26}
I. Vigdorovich.
Selfless reduced $C^*$-algebras of linear groups.
{\it Proc. Lond. Math. Soc. (3)} \textbf{133} (2026), e70180.

\bibitem{Whi22}
S. White.
Abstract classification theorems for amenable $C^*$-algebras.
In {\it Proceedings of the International Congress of Mathematicians 2022, Vol. 4. Sections 5--8}, 3314--3338, EMS Press, Berlin, 2022. 

\bibitem{Win12}
W. Winter.
Nuclear dimension and $\mathcal{Z}$-stability of pure $C^*$-algebras.
{\it Invent. Math.} {\bf 187} (2012), 259--342. 

\bibitem{Win19}
W. Winter.
Structure of nuclear $C^*$-algebras: from quasidiagonality to classification and back again.
In {\it Proceedings of the International Congress of Mathematicians 2018 (Rio de Janeiro), Vol. 3}, 1801--1823. World Sci. Publ., Hackensack, NJ, 2019. 

\end{thebibliography}
\end{document}